%% file: subalsection-v8.tex
\documentclass[11pt]{amsart}
\usepackage{amsthm, amssymb, amsmath, a4wide}

\usepackage{enumitem}
\usepackage{varioref}
\usepackage[bookmarks=false, pagebackref]{hyperref}
\usepackage
{cleveref}
\crefname{thm}{theorem}{theorems}
\crefname{prop}{proposition}{propositions}

\usepackage{xcolor}
\hypersetup{
    colorlinks,
    linkcolor={red!50!black},
    citecolor={blue!50!black},
    urlcolor={blue!80!black}
}

\usepackage{times}

\usepackage{enumitem}

\usepackage{tikz}
\usetikzlibrary{calc} 

\title[Intersection of finitely generated subalgebras also finitely generated?]{When is the intersection of two finitely generated subalgebras of a polynomial ring also finitely generated?}
\author{Pinaki Mondal}
\address{University of The Bahamas}
\email{pinakio@gmail.com}

\DeclareMathOperator{\charac}{characteristic}
\DeclareMathOperator\gal{Gal} 
 
\DeclareMathOperator\trdeg{tr.deg} 

\input{commands-2012}

\newcommand{\adelta}{A^\delta}
\newcommand{\kxdelta}{\kk[X]^\delta}
\newcommand{\kxdeltaone}{\kk[X]^{\delta_1}}
\newcommand{\kxdeltatwo}{\kk[X]^{\delta_2}}

\newcommand{\cxydelta}{\cc[x,y]^{\delta}}
\newcommand{\cxydeltai}{\cc[x,y]^{\delta_i}}

\newcommand{\ldt}{\text{l.d.t.}}

\newcommand{\gplus}[1]{g_{+,#1}}
\newcommand{\gminus}[1]{g_{-,#1}}

\newlist{properties}{enumerate}{2}
\setlist[properties,1]{label=(\Alph*), ref=\Alph*}
\setlist[properties,2]{label=(\arabic*), ref=\thepropertiesi(\arabic*)}
\crefalias{properties}{property} 
\crefalias{properties}{property} 

\newlist{prooflist}{enumerate}{3}
\setlist[prooflist,1]{label=(\roman*)}
\setlist[prooflist,2]{label=(\arabic*)}
\setlist[prooflist,3]{label=(\alph*)}

\newlist{defnlist}{enumerate}{3}
\setlist[defnlist,1]{label=(\alph*)}
\setlist[defnlist,2]{label=(\arabic*), ref=(\alph{defnlisti}.\arabic*)}
\setlist[defnlist,3]{label=(\roman*), ref=(\alph{defnlisti}.\arabic{defnlistii}.\roman*)}

\crefname{bold-question}{question}{questions}

\begin{document}
\maketitle

\begin{abstract} 
We study two variants of the following question: ``Given two finitely generated $\cc$-subalgebras $R_1, R_2$ of $\cc[x_1, \ldots, x_n]$,  is their intersection also finitely generated?'' We show that the smallest value of $n$ for which there is a counterexample is $2$ in the general case, and $3$ in the case that $R_1$ and $R_2$ are integrally closed. We also explain the relation of this question to the problem of constructing algebraic compactifications of $\cc^n$ and to the moment problem on semialgebraic subsets of $\rr^n$. The counterexample for the general case is a simple modification of a construction of Neena Gupta, whereas the counterexample for the case of integrally closed subalgebras uses the theory of normal analytic compactifications of $\cc^2$ via {\em key forms} of valuations centered at infinity. 
\end{abstract}

\section{Introduction} \label{sec-intro}
\begin{bold-question} \label{original-question}
Take two subrings of $\cc[x_1, \ldots, x_n]$ which are finitely generated as algebras over $\cc$. Is their intersection also finitely generated as a $\cc$-algebra? 
\end{bold-question}

The only answer to \cref{original-question} in published literature (obtained via a MathOverflow enquiry \cite{mathoverflow-intersection}) seems to be a class of counterexamples constructed by Bayer \cite{bayer} for $n \geq 32$ using Nagata's counterexample to Hilbert's fourteenth problem from \cite{nagata-fourteenth-lecture} and Weitzenb\"ock's theorem \cite{weitzenbock} on finite generation of invariant rings. After an earlier version of this article appeared on arXiv, however, Wilberd van der Kallen communicated to me a simple counterexample for $n = 3$:

\begin{example} \label{neena-example}
Let $R_1 := \cc[x^2, x^3, y, z]$ and $R$ be the ring of invariants of the action of the additive group $\GG_a := (\cc,+)$ on $R_1$ given by 
\begin{align}
y \mapsto y + x^3,\ z \mapsto z + x^2 \label{neenaction}
\end{align}
Then a result of Bhatwadekar and Daigle \cite{bhaigle} shows that $R$ is not finitely generated over $\cc$. Neena Gupta communicated this construction to Wilberd van der Kallen as an example of a $\GG_a$-action with non-finitely generated ring of invariants. Van der Kallen noted that if $R_2$ is the ring of invariants of the action defined by \eqref{neenaction} of $\GG_a$ on $\cc[x,y,z]$, then $R_2 = \cc[x,y-zx]$ and $R = R_1 \cap R_2$, so that it serves as a counterexample to \cref{original-question}. Indeed, it is straightforward to see directly that $R = \cc[x^\alpha (y - zx)^\beta: (\alpha,\beta) \in S]$, where 
\begin{align*}
S := \{(\alpha, \beta) \in \zz_{\geq 0}^2:\ \text{either}\ \beta = 0\ \text{or} \ \alpha \geq 2\} 
\end{align*}
is a non-finitely generated sub-semigroup of $\zz^2$. 
\end{example}

A variant of \cref{neena-example} in fact gives a counterexample to \cref{original-question} for $n = 2$: 

\begin{example} \label{2-example}
Let $R_1 := \cc[x^2, x^3, y]$ and $R_2 := \cc[x^2, y - x]$. Then $R := R_1 \cap R_2 = \cc[x^{2\alpha} (y-x)^\beta: (\alpha,\beta) \in S']$, where 
\begin{align*}
S' := \{(\alpha, \beta) \in \zz_{\geq 0}^2:\ \text{either}\ \beta = 0\ \text{or} \ \alpha \geq 1\} 
\end{align*}
is a non-finitely generated sub-semigroup of $\zz^2$. 
\end{example}

Since \cref{original-question} holds for $n = 1$ (see e.g.\ assertion \eqref{dim1} of \cref{the-thm}), \cref{2-example} gives a complete answer to \cref{original-question}. In this article we consider a natural variant of \cref{original-question}: denote the subrings of $\cc[x_1, \ldots, x_n]$ in \cref{original-question} by $R_1, R_2$, and their intersection by $R$. 

\begin{bold-question} \label{normal-question}
If $R_1$ and $R_2$ are finitely generated and {\em integrally closed}\footnote{The {\em integral closure} of a subring $R$ of a ring $S$ is the set of all elements $x \in S$ which satisfies an equation of the form $x^d + \sum_{i=1}^d a_ix^{d-i} = 0$ for some $d > 0$ and $a_1, \ldots, a_d \in R$. A domain is {\em integrally closed} if it itself is its integral closure in its field of fractions.} $\cc$-subalgebras of $\cc[x_1, \ldots, x_n]$, is $R$ also finitely generated?
\end{bold-question}

Note that in each of \cref{neena-example,2-example} the ring $R_1$ is not integrally closed, so that they do not apply to \cref{normal-question}. Our findings are compiled in the following theorem. 

\begin{thm} \label{the-thm}
\mbox{}
\begin{enumerate}
\item \label{dim1} If the Krull dimension of $R$ is one (or less), then the answer to  \cref{original-question} is affirmative. In particular, the answers to  \cref{original-question,normal-question} are affirmative for $n=1$.

\item \label{dim2} If the Krull dimension of $R$ is $2$, then the answer to \cref{normal-question} is affirmative. In particular, the answer to \cref{normal-question} is affirmative for $n = 2$. 

\item \label{dim3} There are counterexamples to \cref{normal-question} for $n \geq 3$. 
\end{enumerate}
\end{thm}

%
%

Assertions \eqref{dim1} and \eqref{dim2} follow in a straightforward manner from results of Zariski \cite{zariski-hilbert} and Schr\"oer \cite{schroe-traction}. Assertion \eqref{dim3} is the main result of this article: the subrings $R_1$ and $R_2$ from our examples are easy to construct, and our proof that they are finitely generated is elementary; however the proof of non-finite generation of $R_1 \cap R_2$ uses the theory of {\em key forms} (introduced in \cite{algebraicity}) of valuations centered at infinity on $\cc^2$. \\



Finite generation of subalgebras of polynomial algebras has been well studied, see e.g.\ \cite{gale-subalgebras,nagata-finite-generation,evyatar-zaks,eakin-sub-pol-rings,nagata-subrings-pol-rings,wajnryb,gilmer-heinzer-intermediate-rings,dutta-onoda} and references therein. One of the classical motivations for these studies has been Hibert's fourteenth problem. Indeed, as we have mentioned earlier, Bayer's counterexamples to \cref{original-question} for $n \geq 32$ were based on Nagata's counterexamples to Hilbert's fourteenth problem. Similarly, the construction of \cref{neena-example} is a special case of a result of Bhatwadekar and Daigle \cite{bhaigle} on the ring of invariants of the additive group $(\cc, +)$. Our interest in \cref{original-question,normal-question} however comes from two other aspects: compactifications of $\cc^n$ and the moment problem on semialgebraic subsets of $\rr^n$ - this is explained in \cref{motivection}. 

\begin{remuestion}
What can be said about \cref{original-question,normal-question} if $\cc$ is replaced by an arbitrary field $K$?  
\begin{itemize}
\item Our proof shows that assertions \eqref{dim1} and \eqref{dim2} of \cref{the-thm} remain true in the general case, and assertion \eqref{dim3} remains true if $p := \charac(K)$ is zero. However, we do not know if assertion \eqref{dim3} is true in the case that $p > 0$ - see \cref{positive-remark}. 
\item \Cref{neena-example,2-example} give counterexamples to \cref{original-question} if $p = 0$. However, if $p > 0$, then the ring $R$ would be finitely generated over $K$. Indeed, then $R$ would contain $(y - zx)^p$ in the case of \cref{neena-example} and it would contain $(y-x)^p$ in the case of \cref{2-example}; it would follow that $R_2$ is integral over $R$ and  therefore $R$ is finitely generated over $K$ (\cref{am}). Bayer's \cite{bayer} construction of counterexamples to \cref{original-question} also requires zero characteristic (because of its dependence on Weitzenb\"ock's theorem). In particular, we do not know of a counterexample to \cref{original-question} in positive characteristics. 
\end{itemize}
\end{remuestion}

\subsection{Organization}
In \cref{motivection} we explain our motivations to study \cref{original-question}. In \cref{12section} we prove assertions \eqref{dim1} and \eqref{dim2} of \cref{the-thm}, and in \cref{countersection} we prove assertion \eqref{dim3}. \Cref{counter-thm} gives the general construction of our counterexamples to \cref{normal-question} for $n = 3$, and \cref{counterexample} contains a simple example. \Cref{key-section} gives an informal introduction to {\em key forms} used in the proof of \cref{counter-thm}, and \cref{integral-section} contains the proof of a technical result used in the proof of \cref{counter-thm}.

\subsection{Acknowledgments}
I would like to thank Pierre Milman - the mathematics of this article was worked out while I was his postdoc at University of Toronto. I would also like to thank Wilberd van der Kallen for providing \cref{neena-example}, and the referees for some suggestions which significantly improved the quality of the exposition of this article. The first version of this article has been written up during the stay at the Weizmann Institute as an Azrieli Fellow, and the later versions at the University of the Bahamas. 

\section{Motivation} \label{motivection}
\subsection{Compactifications of affine varieties} \label{degree-like-section}
Our original motivation to study \cref{original-question} comes from construction of projective compactifications of $\cc^n$ via {\em degree-like functions}. More precisely, given an affine variety $X$ over a field $\kk$, a {\em degree-like function} on the ring $\kk[X]$ of regular functions on $X$ is a map $\delta: \kk[X] \to \zz \cup \{-\infty\}$ which satisfies the following properties satisfied by the degree of polynomials:
\begin{prooflist}
\item $\delta(\kk) = 0$,
\item \label{mult-property} $\delta(fg) \leq \delta(f) + \delta(g)$,
\item $\delta(f+g) \leq \max\{\delta(f), \delta(g)\}$.
\end{prooflist} 
The graded ring associated with $\delta$ is
\begin{align}
\kxdelta := \dsum_{d \geq 0} \{f \in \kk[X]: \delta(f) \leq d\} \cong \sum_{d \geq 0} \{f \in \kk[X]: \delta(f) \leq d\}t^d \subseteq \kk[X][t] \label{kxdelta}
\end{align}
where $t$ is an indeterminate. If $\delta$ satisfies the following properties:
\begin{prooflist}[resume]
\item \label{positive} $\delta(f) > 0$ for all non-constant $f$, and 
\item \label{finitely-generated-degree} $\kxdelta$ is a finitely generated $\kk$-algebra,
\end{prooflist}
then $\bar X^\delta := \proj \kxdelta$ is a {\em projective completion} of $X$, i.e.\ $\bar X^\delta$ is a projective (and therefore, {\em complete}) variety that contains $X$ as a dense open subset (see e.g.\ \cite[Proposition 2.5]{sub1}). It is therefore a fundamental problem in this theory to determine if $\kxdelta$ is finitely generated for a given $\delta$. \\

It is straightforward to check that the maximum of finitely many degree-like functions is also a degree-like function, and taking the maximum is one of the basic ways to construct new degree-like functions (see e.g.\ \cite[Theorem 4.1]{sub1}). For example, an $n$-dimensional convex polytope $\scrP \subset \rr^n$ with integral vertices and containing the origin in its interior determines a degree-like function on $\kk[x_1, x_1^{-1}, \ldots, x_n, x_n^{-1}]$ defined as follows: 
$$\delta_\scrP(\sum a_\alpha x^\alpha) := \inf\{d\in \zz: d \geq 0,\ \alpha \in d\scrP\ \text{for all}\ \alpha \in \zz^n\ \text{such that}\ a_\alpha \neq 0\}$$ 
It is straightforward to see that $\delta_\scrP$ satisfies properties \ref{positive} and \ref{finitely-generated-degree}, so that it determines a projective completion $X_\scrP$ of the torus $\nktorus$. It turns out that $X_\scrP$ is precisely the {\em toric variety} corresponding to $\scrP$. Moreover, $\delta_\scrP$ is the maximum of some other `simpler' degree-like functions determined by facets of $\scrP$ - see \cref{sub-picture} for an example.\\

\begin{figure}[h]
\begin{center}
\begin{tikzpicture}[scale=0.6]
	\begin{scope}[shift={(0,0)}]
		\draw [gray,  line width=0pt] (-1.5,-1.5) grid (2.5,2.5);
		\draw [<->] (0,2.5) node (yaxis) [above] {$y$}
       	 |- (2.5,0) node (xaxis) [right] {$x$};
       	 
       	\draw[green,thick ] (-1,-1) -- (2,-1) -- (-1,2) -- cycle;
       	\fill[red] (0,0) circle (4pt);
       	
       	\draw (-0.5,-0.5) node {\textcolor{green}{$\scrP$}};
		\draw (-2.5,-1) node [below right, text width= 4.5cm] {
			\scriptsize{
 			\begin{align*}
 			\delta_{\scrP}: &~ x \mapsto 1, x^{-1} \mapsto 1, y \mapsto 1 \\
 							&~ y^{-1} \mapsto 1, x^{-1}y^{-1} \mapsto 1 \\
 							&~ x^2y^{-1} \mapsto 1, x^{-1}y^2 \mapsto 1
 			\end{align*}
 			}
		};    	

	\end{scope}
	
	\begin{scope}[shift={(7,0)}]
		\draw [gray,  line width=0pt] (-1.5,-1.5) grid (2.5,2.5);
		\draw [<->] (0,2.5) node (yaxis) [above] {$y$}
       	 |- (2.5,0) node (xaxis) [right] {$x$};
       	 
       	 \fill[red] (0,0) circle (4pt);
       	 
       	\draw[green,thick ] (-1,-1) -- (2,-1);

		\draw [blue, thick, ->] (0.5,-1) -- (0.5,-1.5);
		
		\draw (-2.5,-1) node [below right, text width= 4cm] {
			\scriptsize{
 			\begin{align*}
 				\delta_1(x^\alpha y^\beta) 	&= (\alpha, \beta) \cdot (0,-1) \\
 											&= -\beta
 			\end{align*}
 			}
		};   
	\end{scope}
	
	\begin{scope}[shift={(13,0)}]
		\draw [gray,  line width=0pt] (-1.5,-1.5) grid (2.5,2.5);
		\draw [<->] (0,2.5) node (yaxis) [above] {$y$}
       	 |- (2.5,0) node (xaxis) [right] {$x$};
       	 
       	 \fill[red] (0,0) circle (4pt);
       	 
       	\draw[green,thick ] (2,-1) coordinate (p1) -- (-1,2) coordinate (p2);
       	\coordinate (m) at ($(p1)!0.5!(p2)$);
		\draw [blue, thick, ->] (m) -- ($(m)!0.5cm!-90:(p2)$);
		
		\draw (-2.5,-1) node [below right, text width= 4cm] {
			\scriptsize{
 			\begin{align*}
 				\delta_2(x^\alpha y^\beta) 	&= (\alpha, \beta) \cdot (1,1) \\
 											&= \alpha + \beta
 			\end{align*}
 			}
		};

	\end{scope}
	
	\begin{scope}[shift={(19,0)}]
		\draw [gray,  line width=0pt] (-1.5,-1.5) grid (2.5,2.5);
		\draw [<->] (0,2.5) node (yaxis) [above] {$y$}
       	 |- (2.5,0) node (xaxis) [right] {$x$};
       	 
       	 \fill[red] (0,0) circle (4pt);
       	 
       	\draw[green,thick ] (-1,-1) -- (-1,2);

		\draw [blue, thick, ->] (-1,0.5) -- (-1.5,0.5);
		
		\draw (-2.5,-1) node [below right, text width= 4cm] {
			\scriptsize{
 			\begin{align*}
 				\delta_3(x^\alpha y^\beta) 	&= (\alpha, \beta) \cdot (-1,0) \\
 											&= -\alpha
 			\end{align*}
 			}
		};   
	\end{scope}
\end{tikzpicture}
\caption{$\delta_\scrP = \max\{\delta_1, \delta_2, \delta_3\}$}  \label{sub-picture}
\end{center}
\end{figure}

The preceding discussion suggests that the following is a fundamental question in the theory of degree-like functions:

\begin{bold-question} \label{the-degree-like-question}
Let $\delta := \max\{\delta_1, \delta_2\}$. If $\kxdeltaone$ and $\kxdeltatwo$ are finitely generated algebras over $\kk$, is $\kxdelta$ also finitely generated over $\kk$?
\end{bold-question}
In the scenario of \cref{the-degree-like-question}, identifying $\kxdeltaone$ and $\kxdeltatwo$ with subrings of $\kk[X][t]$ as in \eqref{kxdelta} implies that $\kxdelta = \kxdeltaone \cap \kxdeltatwo$. Consequently, in the case that $\kk = \cc$ and $X$ is the affine space $\cc^n$, \cref{the-degree-like-question} is a special case of \cref{original-question}, and
our counterexamples to \cref{normal-question} are in fact counterexamples to this special case with $X = \cc^2$.
%

\subsection{Moment problem}
Given a closed subset $S$ of $\rr^n$, the {\em $S$-moment problem} asks for characterization of linear functionals $L$ on $\rr[x_1, \ldots, x_n]$ such that $L(f)= \int_S f\, \mathrm{d} \mu$ for some (positive Borel) measure $\mu$ on $S$. Classically the moment problem was considered on the real line ($n = 1$): given a linear functional $L$ on $\rr[x]$, a necessary and sufficient condition for $L$ to be induced by a positive Borel measure on $S \subseteq \rr$ was shown to be  
\begin{itemize}
\item $L(f^2 + xg^2) \geq 0$ for all $f,g \in \rr[x]$ in the case that $S = [0, \infty)$ (Stieltjes \cite{stieltjes});
\item $L(f^2) \geq 0$ for all $f\in \rr[x]$ in the case that $S = \rr$ (Hamburger \cite{hamburgerIII});
\item $L(f^2 + xg^2 + (1-x)h^2) \geq 0$ for all $f,g,h \in \rr[x]$ in the case that $S = [-1,1]$ (Hausdorff \cite{hausdorffI}).
\end{itemize}
In the general case Haviland \cite{haviland} showed that $L$ is induced by a positive Borel measure on $S$ iff $L(f) \geq 0$ for every polynomial $f$ which is non-negative on $S$. 
Since sums of squares of polynomials are obvious examples of non-negative on $S$, Haviland's theorem motivates the following definition. 

\begin{defn}[{Powers and Scheiderer \cite{powers-scheiderer}}]
Given a closed subset $S$ of $\rr^n$ and a subset $\scrP$ of $\rr[x_1, \ldots, x_n]$, we say that $\scrP$ {\em solves the $S$-moment problem} if for every linear functional $L$ on $\rr[x_1, \ldots, x_n]$, $L$ is induced by a positive Borel measure on $S$ iff $L(g^2f_1 \cdots f_r) \geq 0$ for every $g \in \rr[x_1, \ldots, x_n]$, $f_1, \ldots, f_r \in \scrP$, $r \geq 0$.
\end{defn}

In particular, the classical examples show that $\emptyset$, $\{x\}$, $\{x, 1 - x\}$ solves the moment problem respectively for $\rr$, $[0, \infty)$, $[0, 1]$. In the case that $S$ is a {\em basic semialgebraic} set, i.e.\ $S$ is defined by finitely many polynomial inequalities $f_ 1 \geq 0, \ldots, f_s \geq 0$, Schm\"udgen \cite{schmudgen} proved that $\{f_1, \ldots, f_s\}$ solves the $S$-moment problem provided $S$ is compact. On the other hand, if $S$ is non-compact, then it may happen that no finite set of polynomials solves the moment problem for $S$ (see e.g.\ \cite{kuhlmann-marshall,powers-scheiderer}). Netzer associated (see e.g.\ \cite[Section 1]{with-tim}) a natural {\em filtration} $\{\scrB_d(S): d \geq 0\}$ on the polynomial ring determined by $S$: 
\begin{align*}
\scrB_d(S) := \{f \in \rr[x_1, \ldots, x_n]: f^2 \leq g\ \text{on}\ S\ \text{for some}\ g \in 	\rr[x_1, \ldots, x_n],\ \deg(g) \leq 2d\}
\end{align*}
In other words, $\scrB_d(S)$ is the set of all polynomials which `grow on $S$ as if they were of degree at most $d$'. The graded algebra corresponding to the filtration is 
\begin{align*}
\scrB(S) := \dsum_{d \geq 0} \scrB_d(S) \cong \sum_{d \geq 0} \scrB_d(S)t^d \subseteq \rr[x_1, \ldots, x_n, t]
\end{align*}
where $t$ is a new indeterminate. 

\begin{thm}[\cite{scheiderer}, Netzer's formulation (appeared in \cite{with-tim})] \label{scheiderer}
If $\scrB_0(S) = \rr$ and $\scrB_d(S)$ is finite dimensional for every $d \geq 0$, then the $S$-moment problem is not solvable. In particular, if $\scrB_0(S) = \rr$ and $\scrB(S)$ is finitely generated as an $\rr$-algebra, then the $S$-moment problem is not solvable. 
\end{thm}

It is straightforward to produce open semialgebraic sets $S$ which satisfies the assumption of \cref{scheiderer}. E.g.\ a {\it standard tentacle} is a set 
$$\left\{ (\lambda^{\omega_1}b_1,\ldots,\lambda^{\omega_n}b_n)\mid \lambda \in \rr, \lambda \geq 1, b\in B\right\}$$ 
where $\omega := (\omega_1, \ldots, \omega_n) \in \zz^n$ and $B\subseteq (\rr\setminus\{0\})^n$ is a compact semialgebraic set with nonempty interior; we call $\omega$ the {\em weight vector} corresponding to the tentacle. If $S$ is a finite union of standard tentacles with weights $\omega_1, \ldots, \omega_k \in \zz^n$, then it is not too hard to see that
\begin{itemize}
\item  $\scrB_0(S) = \rr$ iff the cone $\{\lambda_1 \omega_1 + \cdots + \lambda_k \omega_k : \lambda_1, \ldots, \lambda_k \geq 0\}$ is all of $\rr^n$, and
\item $\scrB(S)$ is finitely generated over $\rr$. 
\end{itemize}
In fact all early examples seemed to suggest that $\scrB(S)$ was finitely generated whenever $\scrB_0(S) = \rr$, at least for {\it regular} semialgebraic sets, i.e.\ sets that are closures of open sets, and it had been asked whether this was indeed the case. In \cite{with-tim} this question had been answered in the negative.
Our construction in \cref{countersection} provides the basis of a particular class of examples in \cite{with-tim} consisting of unions of pairs of (non-standard) tentacles. We now describe the construction. We suggest the reader go over \cref{construction} at this point. \\

\begin{figure}[htp]
\centering
\includegraphics[height=6cm]{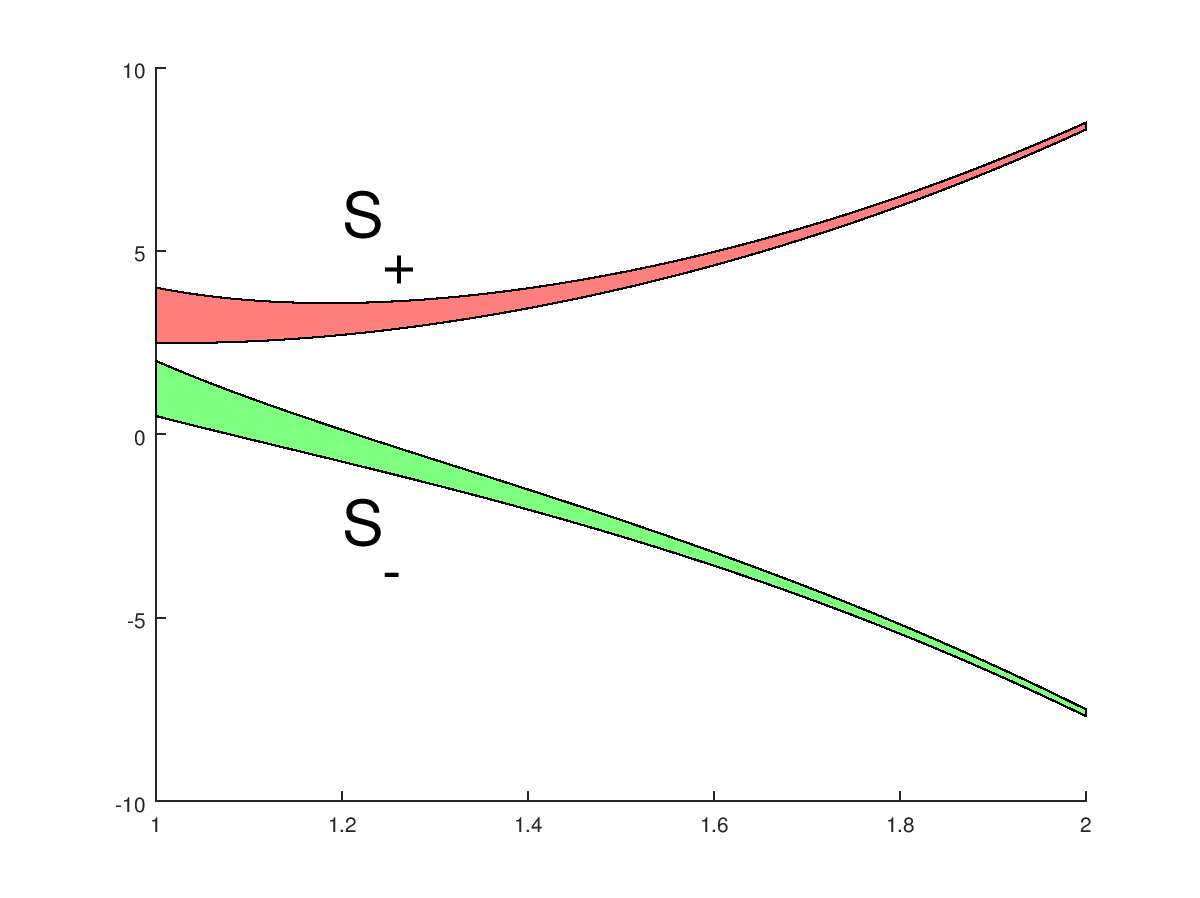}
   
\caption[]{$S = S_+ \cup S_-$, where $S_+ = \{(x,y) \in \rr^2: x \geq 1,\ 0.5 \leq x^3(y - x^3 - x^{-2}) \leq 2\}$ and $S_- = \{(x,y) \in \rr^2: x \geq 1,\ 0.5 \leq x^3(y + x^3 - x^{-2}) \leq 2\}$}
\label{fig:tim}
\end{figure}

Let $p, q_1, \ldots, q_k, \omega_1, \omega_2$ be as in conditions \eqref{odd-property}--\eqref{omega-property} of \cref{construction}. Pick nonzero $a_1, \ldots, a_k \in \rr$ and define $f_+(x), f_-(x)$ as in \eqref{f_+} and \eqref{f_-}. Note that as opposed to \cref{construction}, here $f_+(x)$ and $f_-(x)$ are polynomials over {\em real numbers}. For each $i \in \{+,-\}$, pick positive real numbers $c_{i,1} < c_{i,2}$ and define
\begin{align*} 
S_i := \{(x,y) \in \rr^2: x \geq 1,\ c_{i,1} \leq x^{\omega_2/\omega_1} ( y - f_i(x)) \leq c_{i,2}\}
\end{align*}
Let $\delta_+, \delta_-, R_+, R_-$ be as in \cref{construction}. For $i \in \{+,-\}$, \cite[Lemma 4.3]{with-tim} implies that $f(x,y) \in \scrB_d(S_i)$ iff $\delta_i(f) \leq p\omega_1d$. It follows that the map $\phi: t \mapsto t^{p\omega_1}$ maps $\scrB(S_i) \into R_i$. It is straightforward to check that $R_+, R_ -, R_+ \cap R_-$ are integral over $\scrB(S_+), \scrB(S_-), \scrB(S_+) \cap \scrB(S_-)$ respectively. \Cref{am,counter-thm} then imply that $\scrB(S_+)$ and $\scrB(S_-)$ are finitely generated over $\rr$, but $\scrB(S_+ \cup S_-) = \scrB(S_-) \cap \scrB(S_+)$ is not, even though $\scrB_0(S_+ \cup S_-) = \rr$. \Cref{fig:tim} depicts a pair of $S_+$ and $S_-$ corresponding to \cref{counterexample}. 

%
%
%
%
%
%
%
%


\section{Positive results in dimension at most two} \label{12section}
In this section we prove assertions \eqref{dim1} and \eqref{dim2} of \cref{the-thm}. The proof remains valid if $\cc$ is replaced by an arbitrary algebraically closed field. Moreover, if $k$ is a field with algebraic closure $\bar k$, then a subring $R$ of $k[x_1, \ldots, x_n]$ is finitely generated over $k$ iff $R \otimes_k \bar k$ is finitely generated over $\bar k$; this, together with the preceding sentence, implies that assertions \eqref{dim1} and \eqref{dim2} of \cref{the-thm} remain true if $\cc$ is replaced by an arbitrary field. We use the following results in this section.

\begin{lemma}[{\cite[Corollary 5.22]{am}}] \label{am0}
Let $A$ be a subring of a field $K$. Then the integral closure of $A$ in $K$ is the intersection of all valuation rings in $K$ containing $A$.
\end{lemma}

\begin{lemma}[{\cite[Proposition 7.8]{am}}] \label{am}
Let $A \subseteq B \subseteq C$ be rings such that $A$ is Noetherian, $C$ is finitely generated as an $A$-algebra, and $C$ is integral over $B$. Then $B$ is finitely generated as an $A$-algebra. 
\end{lemma}

\begin{thm}[{\cite{zariski-hilbert}}] \label{zl}
Let $L$ be a field of transcendence degree at most two over a field $k$ and $R$ be an integrally closed domain which is finitely generated as a $k$- algebra. Then $L \cap R$ is a finitely generated $k$-algebra.
\end{thm}

\begin{thm}[{\cite[Corollary 6.3]{schroe-traction}}] \label{schroer}
Let $U$ be a (not necessarily proper) surface (i.e.\ $2$-dimensional irreducible separated scheme of finite type) over a field $k$. Assume $U$ is normal. Then $\Gamma(U,  \sheaf_U)$ is a finitely generated $k$-algebra of dimension $2$ or less. 
\end{thm}

Recall the notation from \cref{the-thm}. In this section we write $L$ for the field of fractions of $R$ and $\bar R_j$ for the integral closure of $R_j$ in its field of fractions, $j = 1, 2$. Moreover, we write  $R'_j := \bar R_j \cap L$, $j = 1,2$. 

\subsection{Proof of  assertion \eqref{dim1} of \cref{the-thm}}
Assume w.l.o.g.\ $\trdeg_\cc(L) = 1$. \Cref{zl} implies that $R'_1$ is finitely generated as a $\cc$-algebra. Let $C$ be the unique non-singular projective curve over $\cc$ such that the field of rational functions on $C$ is $L$. Then $C'_1  := \spec R'_1$ is isomorphic to $C \setminus \{x_1, \ldots, x_k\}$ for finitely many points $x_1, \ldots, x_k \in C$. Then the local rings $\sheaf_{C,x_j}$ of $C$ at $x_j$'s are the only one dimensional valuation rings of $L$ not containing $R'_1$. Let $\bar R$ be the integral closure of $R$ in $L$. Since $\bar R \subseteq R'_1$, \cref{am0} implies that
$$\bar R = R'_1 \cap \sheaf_{C,x_{j_1}} \cap \cdots \cap \sheaf_{C,x_{j_s}}$$
for some $j_1, \ldots, j_s \in \{1, \ldots, k\}$. Then $\bar R$ is the ring of regular functions on $C \setminus \{x_j: j\not\in \{j_1, \ldots, j_s\}\}$, and is therefore finitely generated over $\cc$. \Cref{am} then implies that $R$ is finitely generated over $\cc$. 
\qed 

\subsection{Proof of  assertion \eqref{dim2} of \cref{the-thm}}
Let $L$ be the field of fraction of $R$. Due to assertion \eqref{dim1} we may assume $\trdeg_\cc(L) = 2$. \Cref{zl} implies that $R'_1$ and $R'_2$ are finitely generated over $\cc$. Let $X_i := \spec R'_i$ and $\bar X_i$ be a projective compactification of $X_i$, $i = 1,2$. Let $\bar X$ be the closure in $\bar X_1 \times \bar X_2$ of the graph of the birational correspondence $X_1 \dashrightarrow X_2$ induced by the identification of their fields of rational functions, and $\tilde X$ be the normalization of $\bar X$. For each $i$, let $\pi_i: \tilde X \to \bar X_i$ be the natural projection and set $U_i := \pi_i^{-1}(X_i)$. 
\begin{claim} \label{R'_i-claim}
$R'_i = \Gamma(U_i, \sheaf_{\tilde X})$.
\end{claim}
\begin{proof}
Clearly $R'_i \subseteq \Gamma(U_i, \sheaf_{\tilde X})$. For the other inclusion, pick $f \in \Gamma(U_i, \sheaf_{\tilde X})$. Since $R'_i$ is integrally closed, it suffices to show that $f$ is bounded near every point of $X_i$. Indeed, if $x \in X_i$, then $f$ is regular on $\pi_i^{-1}(x)$, and is therefore constant on all positive dimensional connected components of $\pi_i^{-1}(x)$.
\end{proof}

The assumption that $R_i$'s are integrally closed together with \cref{R'_i-claim,schroer} imply that $R = R'_1 \cap R'_2 = \Gamma(U_1 \cup U_2, \sheaf_{\tilde X})$ is finitely generated over $\cc$, as required. 
\qed

\section{Counterexamples in dimension three}  \label{countersection}
In this section we prove assertion \eqref{dim3} of \cref{the-thm}. In \cref{construction} we describe the construction of counterexamples to \cref{normal-question} for $n =3$, and in \cref{proof1section,proof2section} we prove that these satisfy the required properties. 

\subsection{Construction of the counterexamples} \label{construction}
Let $p, q_1, \ldots, q_k$ be integers such that 
\begin{properties}
\item \label{odd-property} $p$ is an odd integer $\geq 3$, 
\item $0 \leq q_1 < q_2 < \cdots < q_k < p$, 
\item \label{even} there exists $j$ such that $q_j$ is positive and even,
\end{properties}
and let $\omega_1,\omega_2$ be relatively prime positive integers such that 
\begin{properties}[resume]
\item \label{omega-property} $p \geq \omega_2/\omega_1 > q_k$. 
\end{properties}
%
%
Pick nonzero $a_1, \ldots, a_k \in \cc$ and set
\begin{align}
f_+(x) &:=x^p +  \sum_{j=1}^k a_j x^{-q_j} \label{f_+}\\
f_-(x) &:= f_+(-x) = - x^p +  \sum_{j=1}^k (-1)^{q_j} a_j  x^{-q_j}  \label{f_-}
\end{align}
For each $i\in \{+,-\}$, let $y_i := y - f_i(x)$ and $\delta_i$ be (the restriction to $\cc[x,y]$ of) the {\em weighted degree}\footnote{See \cref{weighted-section} for a discussion of weighted degrees.} on $\cc(x,y) = \cc(x,y_i)$ corresponding to weights $\omega_1$ for $x$ and $-\omega_2$ for $y_i$, and 
\begin{align}
R_i := \cxydeltai = \sum_{d \geq 0} \{g \in \cc[x,y]: \delta_i(g) \leq d\}t^d \subseteq \cc[x,y,t]
\end{align}

Assertion \eqref{dim3} of \cref{the-thm} follows from \cref{counter-thm} below.

\begin{thm} \label{counter-thm}
\mbox{}
\begin{enumerate}
\item \label{finite-type-assertion} $R_+$ and $R_-$ are finitely generated integrally closed $\cc$-algebras. 
\item \label{infinite-type-assertion}  $R_+ \cap R_-$ is not finitely generated over $\cc$. 
\end{enumerate}
Let $\Delta_d := \{f \in \cc[x,y]: \delta_i(f) \leq d, \ i = 1, 2\} $, so that $R_+ \cap R_- = \sum_{d \geq 0} \Delta_d t^d$. Then 
\begin{enumerate}
\setcounter{enumi}{2}
\item \label{zero-assertion} $\Delta_0 = \cc$.
\item \label{finite-assertion} If $\omega_2/\omega_1 < p$, then each $\Delta_d$ is finite dimensional (as a vector space) over $\cc$. 
\item \label{infinite-assertion} If $\omega_2/\omega_1 = p$, then there exists $d > 0$ such that $\Delta_d$ is infinite dimensional (as a vector space) over $\cc$.
\end{enumerate}
\end{thm}

\begin{example} \label{counterexample}
Take $f_+ = x^3 + x^{-2}$ and $\omega_2/\omega_1 = 3$. Then $f_- = -x^3 + x^{-2}$. Let 
\begin{align*}
\gplus{0} &:= y - x^3,\ 			& \gminus{0} &:= y+x^3 \\
\gplus{1} &:= x^2(y-x^3)		& \gminus{1} &:= x^2(y+x^3) 
\end{align*}
Let $\scrG$ be a (finite) set of generators of the subsemigroup 
$$\{(\alpha, \beta_0, \beta_1, d) \in (\zz_{\geq 0})^4: \alpha -2\beta_0 - \beta_1 \leq d\}$$
of $\zz^4$. \Cref{deltaplus-cor} below shows that 
\begin{align*}
R_+ &= \cc[x^\alpha \gplus{0}^{\beta_0} \gplus{1}^{\beta_1}t^d : (\alpha, \beta_0, \beta_1, d)  \in \scrG]\\
R_- &= \cc[x^\alpha \gminus{0}^{\beta_0} \gminus{1}^{\beta_1}t^d : (\alpha, \beta_0, \beta_1, d)  \in \scrG]
\end{align*}
On the other hand assertions \eqref{zero-assertion} and \eqref{infinite-assertion} of \cref{counter-thm} imply that $\Delta_0 = \cc$ but $\Delta_d$ is infinite dimensional over $\cc$ for some $d \geq 1$; in particular, $R_+ \cap R_-$ is not finitely generated. 
\end{example}

\begin{rem} \label{positive-remark}
Our proof of \cref{counter-thm} remains correct if $\cc$ is replaced by an algebraically closed field $\kk$ of characteristic zero. However, if $\kk$ has positive characteristic, we can only say the following:
\begin{defnlist}
\item $R_+$ and $R_-$ remain finitely generated integrally closed $\kk$-algebras (our proof for assertion \eqref{finite-type-assertion} of \cref{counter-thm} remains valid);
\item \label{positive-finite} if $\omega_2/\omega_1 < p$, then $\Delta_0 = \kk$ and each $\Delta_d$ is a finite dimensional vector space over $\kk$ (in the case $\omega_2/\omega_1 < p$, assertions \eqref{zero-assertion} and \eqref{finite-assertion} of \cref{counter-thm} are essentially consequences of \cite[theorem 1.4]{sub2-1}, which in turn is a consequence of computations of intersection numbers of curves at infinity on certain {\em completions} (i.e.\ compactifications in the analytic topology) of $\cc^2$; the intersection numbers remain unchanged if $\cc$ is replaced by an arbitrary algebraically closed field $\kk$).
\item if $\omega_2/\omega_1 < p$ and $a_1, \ldots, a_k$ are contained in the algebraic closure of a finite field, then $R_+ \cap R_-$ is finitely generated over $\kk$ (this is a consequence of the `explanation' in parentheses of assertion \ref{positive-finite} and Artin's result (see e.g.\ \cite[Theorem 14.21]{badescu}) that every two dimensional algebraic space over algebraic closures of finite fields are quasi-projective surfaces). In particular, in this case our construction does {\em not} produce a counterexample to \cref{normal-question}.
\item In the remaining cases we do not know if any of assertions \eqref{infinite-type-assertion}--\eqref{infinite-assertion} of \cref{counter-thm} is true (since our main tool, namely \cite[Theorem 4.1]{algebraicity}, does not apply). 
%
%
%
%
%
%
\end{defnlist}
\end{rem}

\newcommand{\Rplus}{R_+}
\newcommand{\deltaplus}{\delta_+}
\newcommand{\omegaplus}{\omega_+}
\newcommand{\omegapluss}[1]{\omega_{+,#1}}
\newcommand{\piplus}{\pi_+}
\newcommand{\zplus}[1]{z_{+,#1}}

\subsection{Proof of assertion \eqref{finite-type-assertion} of \cref{counter-thm}.} \label{proof1section}
We prove assertion \eqref{finite-type-assertion} of \cref{counter-thm} only for $R_+$, since the statement for $R_-$ follows upon replacing each $a_j$ to $(-1)^{q_j}a_j$. \\

The fact that $R_+$ is integrally closed follows from the observation that $\delta_+(g^k) = k\delta_+(g)$ for each $g \in \cc[x,y]$ and $k \geq 0$, i.e.\ $\delta_+$ is a {\em subdegree} in the terminology of \cite{pisis} (see e.g.\ \cite[Proposition 2.2.7]{pisis}). We give a proof here for the sake of completeness.

\begin{lemma} \label{integrally-closed}
Let $\kk$ be a field and $\eta$ be a degree-like function on a $\kk$-algebra $A$ such that $\eta(g^k) = k\eta(g)$ for each $g \in A$ and $k \geq 0$. If $A$ is an integrally closed domain, then so is $A^\eta$. 
\end{lemma}

\begin{proof}
Let $t$ be an indeterminate. Identify $A^\eta$ with a subring of $A[t]$ as in \eqref{kxdelta}. Then the field of fractions of $A^\eta$ is $K(t)$ where $K$ is the field of fractions of $A$. Let $h \in K(t)$ be integral over $A^\eta$. Since the degree in $t$ gives $A^\eta$ the structure of a graded ring, and since $A[t]$ is an integrally closed overring of $A^\eta$, we may w.l.o.g.\ assume that $h = h't^d$ for some $h' \in A$ and $d \geq 0$. Consider an integral equation of $h$ over $A^\eta$: 
\begin{align*}
(h't^d)^k + \sum_{j = 1}^k f_jt^{i_j} (h't^d)^{k-j} = 0
\end{align*}
where $f_jt^{i_j} \in A^\eta$ for each $j$. Consequently we may assume that $i_j = dj$ for each $j = 1, \ldots, k$ such that $f_j \neq 0$, and therefore 
\begin{align*}
h'^k = - \sum_{j = 1}^k f_jh'^{k-j}
\end{align*} 
It follows that 
\begin{align*}
k\eta(h')
	= \eta(h'^k)  
	&\leq \max\{\eta(f_jh'^{k-j}): 1 \leq j \leq k\} \\
	&\leq\{\eta(h'^{k-j}) + \eta(f_j): 1 \leq j \leq k\} \\
	&\leq \max\{(k-j)\eta(h') + dj: 1 \leq j \leq k\} 
\end{align*}
where the last inequality follows from the definition of $A^\eta$ and the observation that $f_jt^{dj} = f_jt^{i_j} \in A^\eta$. It follows that $\eta(h') \leq d$, which implies that $h = h't^d \in A^\eta$, as required to prove that $A^\eta$ is integrally closed. 
\end{proof}

\Cref{integrally-closed} shows that $\Rplus$ is integrally closed. Now we show that $\Rplus$ is finitely generated over $\cc$. Set $q_0 := 0$ and define
\begin{align}
\gplus{0} &:= y -   x^p\\
\gplus{j} &:= x^{q_j}(y - x^p -  \sum_{i=1}^j a_{i} x^{-q_i}),\ 1 \leq j \leq k,\\
\omegapluss{j} &:= \deltaplus(\gplus{j}) 
	= 
	\begin{cases}
	-\omega_1(q_{j+1} - q_j) &\text{if}\ 0 \leq j \leq k-1, \\
	-\omega_2 &\text{if}\ j = k.
	\end{cases}
\end{align}
Note that $\gplus{j} \in \cc[x,y]$ for each $j = 0, \ldots, k$. Let $z_0, \ldots, z_k$ be indeterminates, $S := \cc[x,z_0, \ldots, z_k]$, and $\omegaplus$ be the weighted degree on $S$ corresponding to weights $\omega_1, \omegapluss{0}, \ldots, \omegapluss{k}$ to respectively $x, \zplus{0}, \ldots, \zplus{k}$. Let $\piplus: S \to \cc[x,y]$ be the map that sends $x \mapsto x$ and $z_j \to \gplus{j}$, $0 \leq j \leq k$. Note that 
\begin{align}
\omegaplus(F) \geq \deltaplus(\piplus(F)) \label{omegadeltaplus}
\end{align}
for each $F \in S$. Let $J_+$ be the ideal in $S$ generated by all weighted homogeneous (with respect to $\omegaplus$) polynomials $F \in S$ such that $\omegaplus(F) > \deltaplus(\piplus(F))$. Note that for each $j = 0, \ldots, k-1$,
$$z_jx^{q_{j+1} - q_j} - a_{j+1} \in J_+$$

\begin{claim} \label{J_+_claim}
$J_+$ is a prime ideal of $S$ generated by $z_jx^{q_{j+1} - q_j} - a_{j+1}$, $0 \leq j \leq k-1$.
\end{claim}

\begin{proof}
The fact that $J_+$ is prime is a straightforward consequence of inequality \eqref{omegadeltaplus} and the observations that both $\omegaplus$ and $\deltaplus$ satisfy property \ref{mult-property} of degree-like functions (see \cref{degree-like-section}) with exact {\em equality}. Let $\tilde J_+$ be the ideal of $S$ generated by $z_jx^{q_{j+1} - q_j} - a_{j+1}$, $0 \leq j \leq k-1$. Then  
\begin{align}
S/\tilde J_+\cong \cc[x,x^{-(q_1-q_0)}, \ldots, x^{-(q_k - q_{k-1})}, z_k] \label{S/J_+}
\end{align}
where the isomorphism is that of {\em graded rings}, the grading on both rings being induced by $\omegaplus$. This implies that $\tilde J_+$ is a prime ideal contained in $J_+$. Since $J_+/\tilde J_+$ is a prime {\em homogeneous} (with respect to the grading) of $S/\tilde J_+$, it follows that if $J_+\supsetneqq \tilde J_+$, then $J_+$ contains an element of the form $x^r z_k^s - \alpha$ for some $\alpha \in \cc$ and $(r,s) \in (\zz_{\geq 0})^2 \setminus \{(0,0)\}$. Since this is impossible by definition $J_+$, it follows that $J_+ = \tilde J_+$, as required.
\end{proof}

\begin{claim} \label{surjective-corollary}
For each $f \in \cc[x,y]$, there exists $F \in S$ such that $\piplus(F) = f$ and $\omegaplus(F) = \deltaplus(f)$.
\end{claim}

\begin{proof}
Let $f \in \cc[x,y]$ and $F \in S$ such that $\piplus(F) = f$. Inequality \eqref{omegadeltaplus} implies that $\omegaplus(F) \geq \deltaplus(f)$. Assume w.l.o.g.\ $\omegaplus(F) > \deltaplus(f)$. It suffices to show that there exists $F' \in S$ such that $\piplus(F') = f$ and $\omegaplus(F') < \omegaplus(F)$. Indeed, if $H$ is the leading weighted homogeneous form (with respect to $\omegaplus$) of $F$, then $H \in J_+$. \Cref{J_+_claim} then implies that
$$H =\sum_{j=0}^{k-1} ( z_jx^{q_{j+1} - q_j} - a_{j+1})H_j$$
for some weighted homogeneous $H_0, \ldots, H_{k-1} \in S$. Setting 
$$F' := (F - H) + \sum_{j=0}^{k-1} H_j z_{j+1}$$ 
does the job.
\end{proof}

\begin{cor} \label{deltaplus-cor}
Let $\Gamma := \{(\alpha, \beta_0, \ldots, \beta_k,d) \in (\zz_{\geq 0})^{k+3}: \alpha\omega_1 + \sum_{j=0}^k \beta_j\omegapluss{j} \leq d\}$. Then 
$\cc[x,y]^{\deltaplus} = \cc[x^\alpha \gplus{0}^{\beta_0} \cdots \gplus{k}^{\beta_k}t^d: (\alpha, \beta_0, \ldots, \beta_k,d) \in \Gamma]$. \qed
\end{cor}

Since $\Gamma$ is a finitely generated subsemigroup of $\zz^{k+3}$, \cref{deltaplus-cor} proves assertion \eqref{finite-type-assertion} of \cref{counter-thm}. \qed

\subsection{Proof of assertions \eqref{infinite-type-assertion}--\eqref{infinite-assertion} of \cref{counter-thm}.}
\label{proof2section}
Let $u,v, \xi$ be indeterminates. Let 
\begin{align}
\phi(u,\xi) 
	:= f_+(u^{1/2}) + \xi u^{-\omega_2/(2\omega_1)} 
	=  u^{p/2} + \sum_{j=1}^k a_ju^{-q_j/2} +  \xi u^{-\omega_2/(2\omega_1)}
\end{align}
and $\eta$ be the degree-like function on $\cc[u,v]$ defined as follows:
\begin{align*}
\eta(g(u,v)) = 2\omega_1\deg_u\left(g(u,v)|_{v = \phi(u,\xi)}\right).
\end{align*}
Now consider the map $\cc[u,v] \into \cc[x,y]$ given by $u \mapsto x^2$ and $v \mapsto y$. It is not hard to check that for each $i \in \{+,-\}$, $\delta_i$ is an {\em extension} of $\eta$, i.e.\ $\delta_i$ restricts to $\eta$ on $\cc(u,v)$. Note that $-\eta$ is a {\em discrete valuation} on $\cc[u,v]$ and $-\delta_+, -\delta_-$ are discrete valuations on $\cc(x,y)$. Since the degree of the extension $\cc(x,y)$ over $\cc(u,v)$ is $2$, it follows from \cite[Theorem VI.19]{zsII} that $\delta_1$ and $\delta_2$ are in fact the {\em only} extensions of $\eta$ to $\cc[x,y]$. Let 
$$\delta := \max\{\delta_+, \delta_-\}.$$
\Cref{semi-integral-extension} then implies that $\cc[x,y]^\delta$ is integral over $\cc[u,v]^\eta$.\\

Now note that 
\begin{align*}
v^2|_{v = \phi(u,\xi)} 
	&= u^p +  2a_1u^{(p-q_1)/2} + \cdots + 2a_ku^{(p-q_k)/2} +  2\xi u^{p/2 - \omega_2/(2\omega_1)} + \ldt 
\end{align*}
where $\ldt$ denotes terms with degree in $u$ smaller than 
$$\epsilon := p/2 - \omega_2/(2\omega_1)$$
Note that $\epsilon \geq 0$ due to defining property \eqref{omega-property} of $\omega_1, \omega_2$. Define
\begin{align}
h_j &=
	\begin{cases}
	v^2 - u^p &\text{if}\ j = 0,\\
	h_{j-1} - 2a_ju^{-q_j/2}v &\text{if}\ 1 \leq j \leq k\ \text{and $q_j$ is even,} \\
	h_{j-1} - 2a_ju^{(p-q_j)/2} &\text{if}\ 1 \leq j \leq k\ \text{and $q_j$ is odd.} 
	\end{cases}
\end{align}
It is straightforward to verify that 
\begin{defnlist}
\item \label{h-property-0} $\eta(h_0) > \eta(h_1) >  \cdots > \eta(h_k) = 2\omega_1\epsilon \geq 0$. 
\item $h_k|_{v = \phi(u,\xi)} = 2\xi u^\epsilon + $ terms with degree in $u$ smaller than $\epsilon$.  
\end{defnlist}
It then follows that $u, v, h_0, \ldots, h_k$ is the sequence of {\em key forms} of $\eta$ - see \cref{key-section} for an informal discussion of key forms, and \cite[definition 3.16]{algebraicity} for the precise definition. Property \eqref{even} of $q_1, \ldots, q_k$ implies that $h_k$ is {\em not} a polynomial. This, together with observation \ref{h-property-0} and \cite[theorem 4.13 and proposition 4.14]{with-tim} implies (see \cref{key-results})  that
\begin{defnlist}[resume]
\item \label{finite-type-eta}$\cc[u,v]^\eta$ is not finitely generated over $\cc$,
\item $\eta(f) > 0$ for each $f \in \cc[u,v] \setminus \cc$,
\item if $\epsilon > 0$, then $\{f \in \cc[u,v]: \eta(f) \leq d \}$ is a finite dimensional vector space over $\cc$ for all $d \geq 0$, 
\item \label{infinite-eta} if $\epsilon = 0$, then there exists $d > 0$ such that $\{f \in \cc[u,v]: \eta(f) \leq d \}$ is an infinite dimensional vector space over $\cc$.
\end{defnlist}

Since $R = \cxydelta$ is integral over $\cc[u,v]^\eta$, observations \ref{finite-type-eta}--\ref{infinite-eta} imply assertions \eqref{infinite-type-assertion}--\eqref{infinite-assertion} of \cref{counter-thm}.
\qed

\appendix
\section{Key forms: an informal introduction} \label{key-section}
\subsection{} \label{weighted-section} The simplest of the degree-like functions on $\cc[x,y]$ are {\em weighted degrees}: given a pair of relatively prime integers $(\omega_1, \omega_2) \in \zz^2$, the corresponding {\em weighted degree} $\omega$ is defined as follows: 
\begin{align*}
\omega(\sum_{\alpha, \beta}c_{\alpha,\beta}x^{\alpha}y^\beta)
	&:= \max\{\alpha\omega_1 + \beta\omega_2: c_{\alpha,\beta} \neq 0\}
\end{align*}

Assume $\omega_1$ and $\omega_2$ are positive. Then the weighted degree $\omega$ can also be described as follows: take the one dimensional family of curves $C_\xi := \{(x,y): y^{\omega_1} - \xi x^{\omega_2} = 0\}$ parametrized by $\xi \in \cc$. Each of these curves has {\em one place at infinity}, i.e.\ its closure in $\pp^2$ intersects the line at infinity on $\pp^2$ at a single point, and the germ of the curve is analytically irreducible at that point. Then for each $f \in \cc(x,y)$, $\omega(f)$ is simply the pole of $f|_{C_\xi}$ at the unique point at infinity on $C_\xi$ for generic $\xi \in \cc$. 

\subsection{} \label{depth-2-section} Now consider the family of curves $D_\xi := \{(x,y): y^2 - x^3 - \xi x^2 = 0\}$, again parametrized by $\xi \in \cc$. Each $D_\xi$ also has one place at infinity, and therefore defines a degree-like function $\eta$ on $\cc(x,y)$ defined as in the preceding paragraph: $\eta(f)$, where $f$ is a polynomial, is the pole of $f|_{D_\xi}$ at the unique point at infinity on $D_\xi$ for generic $\xi \in \cc$. Then it is not hard to see that 
\begin{itemize}
\item $\eta(x) = 2$, $\eta(y) = 3$, $\eta(y^2 - x^3) = 4$,
\item Given an expression of the form 
\begin{align}
f = \sum_{\alpha_0, \alpha_1, \alpha_2} c_\alpha x^{\alpha_0}y^{\alpha_1}(y^2 - x^3)^{\alpha_2}
\end{align}
where $0 \leq \alpha_1 < 2$ and $\alpha_2 \geq 0$, one has
\begin{align*}
\eta(f) = \max\{ 2\alpha_0 + 3\alpha_1 + 4\alpha_2: c_\alpha \neq 0\}
\end{align*}
\end{itemize}

\subsection{} \label{general-section}  Both $\omega$ from \ref{weighted-section} and $\eta$ from \ref{depth-2-section} are {\em divisorial semidegrees} on $\cc[x,y]$ - these are degree-like functions $\delta$ on $\cc[x,y]$ such that there is an algebraic compactification $\bar X$ of $\cc^2$ and an irreducible curve $E \subseteq \bar X \setminus \cc^2$ such that for each $f \in \cc[x,y]$, $\delta(f)$ is the pole of $f$ along $E$. For a divisorial semidegree $\delta$, starting with $g_0 := x, g_1 := y$, one can successively form a finite sequence of elements $g_0, \ldots, g_{l+1} \in \cc[x,x^{-1},y]$, $l \geq 0$, such that 
\begin{itemize}
\item for each $i = 1, \ldots, l$, $g_{i+1}$ is a simple `binomial' in $g_0, \ldots, g_i$,
\item $\delta(g_{i+1})$ is smaller than its `expected value', and
\item every polynomial $f$ in $(x,y)$ has an expression in terms of $g_0, \ldots, g_{l+1}$ such that $\delta(f)$ can be computed from that expression from only the knowledge of $\delta(g_0), \ldots, \delta(g_{l+1})$. 
\end{itemize}
The key forms of weighted degrees are simply $x,y$, and the key forms of $\eta$ from \ref{depth-2-section} are $x, y, y^2 - x^3$ (since $\eta(x) = 2$ and $\eta(y) = 3$, the `expected value' of $\eta(y^2 - x^3)$ should have been $6$, whereas its actual value is $4$). 

\subsection{} \label{key-results} A lot of information of a divisorial semidegree $\delta$ can be recovered from its key forms. The results that we use in the proof of \cref{counter-thm} follow from \cite[theorem 4.13 and proposition 4.14]{with-tim}, and are as follows: if $g_{l+1}$ is the last key form of $\delta$, then
\begin{prooflist}
\item The following are equivalent:
\begin{prooflist}
\item $\delta(f) > 0$ for every non-constant polynomial $f$ on $\cc[x,y]$,
\item either $\delta(g_{l+1}) > 0$, or $\delta(g_{l+1}) = 0$ and $g_{l+1}$ is {\em not} a polynomial.
\end{prooflist}
\item The following are equivalent:
\begin{prooflist}
\item $\cxydelta$ is {\em not} finitely generated over $\cc$,
\item $\delta(g_{l+1}) \geq 0$ and $g_{l+1}$ is {\em not} a polynomial.
\end{prooflist}
\item Assume $\cxydelta$ is {\em not} finitely generated over $\cc$. 
\begin{prooflist}
\item if $\delta(g_{l+1}) > 0$, then $L_d := \{f \in \cc[x,y]: \delta(f) \leq d\}$ is a finite dimensional vector space over $\cc$ for each $d \geq 0$. 
\item if $\delta(g_{l+1}) =0$, then there exists $d > 0$ such that $L_d$ is infinite dimensional over $\cc$.
\end{prooflist}
\end{prooflist}

\section{Integral closure of the graded ring of a degree-like function} \label{integral-section}

\begin{defn}
Let $\kk$ be a field and $A$ be a $\kk$-algebra. A degree-like function $\delta$ on $A$ is called a {\em semidegree} if $\delta$ satisfies condition \ref{mult-property} of degree-like functions (see \cref{degree-like-section}) always with an {\em equality}. We say that $\delta$ is a {\em subdegree} if there are finitely many semidegrees $\delta_1, \ldots, \delta_k$ such that for all $f \in A\setminus\{0\}$, 
\begin{align} \label{gensgf-condition}
	\delta(f) = \max\{\delta_1(f), \ldots, \delta_k(f)\}
\end{align}
\end{defn}

\begin{rem}
If $\delta$ is integer-valued on $A\setminus \{0\}$ (i.e.\ $\delta(f) = -\infty$ iff $f = 0$), then $\delta$ is a semidegree iff $-\delta$ is a {\em discrete valuation}. 
\end{rem}

Let $A \subseteq B$ be $\kk$-algebras which are also integral domains. Assume $B$ is integral over $A$ and the quotient field $L$ of $B$ is a finite separable extension of the quotient field $K$ of $A$.

\begin{lemma} \label{semi-integral-extension}
Let $\delta_1, \ldots, \delta_m$ be semidegrees on $A$ which are integer-valued on $A\setminus\{0\}$, and $\delta := \max\{\delta_1, \ldots, \delta_m\}$. For each $i$, $1 \leq i \leq m$, let $\eta_{ij}$, $1 \leq j \leq m_i$, be the extension of $\delta_i$ to $B$. Define $\eta  := \max\{\eta_{ij}, 1 \leq i \leq m, 1 \leq j \leq m_i\}$. If $A$ is integrally closed, then $B^\eta$ is integral over $\adelta$. If in addition $B$ is integrally closed, then $B^\eta$ is the integral closure of $\adelta$ in the quotient field of $B^\eta$. 
\end{lemma}

\begin{proof}
By construction, the restriction of $\eta$ to $A$ is precisely $\delta$, so that $\adelta \subseteq B^\eta$. The last assertion of the lemma follows from the first by \cref{integrally-closed}. We now demonstrate the first assertion. Let $t$ be an indeterminate. Identify $B^\eta$ with a subring of $B[t]$ as in \eqref{kxdelta}. Let $f \in B \setminus \{0\}$ and $d' := \eta(f)$. It suffices to show that $ft^{d'} \in B^\eta$ satisfies an integral equation over $A^\delta$. Let the {\em minimal polynomial} of $f$ over $K$ be
\begin{align}
P(T) := T^d + \sum_e g_e T^{d-e} \label{P(t)}
\end{align}
and $L'$ be the Galois closure of $L$ over $K$. Since $L'$ is Galois over $K$, it contains all the roots $f_1, \ldots, f_d$ of $P(T)$. Since $L/K$ is finite and separable, each $f_i = \sigma_i(f)$ for some $\sigma_i \in \gal(L'/K)$. For each $i,j$, $1 \leq i \leq m$ and $1 \leq j \leq m_i$, let $\{\eta'_{ijk}: 1 \leq k \leq l_{ij}\}$ be the extensions of $\eta_{ij}$ to $L'$. Define 
\begin{align*}
\eta'_i &:= \max\{\eta'_{ijk}: 1 \leq j \leq m_i,\ 1 \leq k \leq l_{ij}\},\ 1 \leq i \leq m,\ \text{and}\\
\eta' 	&:= \max\{\eta'_i : 1 \leq i \leq m\}.
\end{align*}  
Since each of $\delta_i$, $\eta_{ij}$ and $\eta'_{ijk}$'s is the negative of a discrete valuation, it follows that each $\eta'_{ijk} = \eta'_{i11} \circ \sigma_{ijk}$ for some $\sigma_{ijk} \in \gal(L'/K)$ \cite[Theorem VI.12, Corollary 3]{zsII}. It follows that for all $i, j$,
\begin{align*}
\eta'_i(f_j) &= \max\{\eta'_{i11} \circ (\sigma_{ij'k'} \circ \sigma_j)(f): 1 \leq j' \leq m_i,\ 1 \leq k' \leq l_{ij} \}\\
			 &= \max\{\eta'_{ij''k''}(f): 1 \leq j'' \leq m_i,\ 1 \leq k'' \leq l_{ij} \}\\
			 &= \eta'_i(f).
\end{align*} 
Note that $\eta'_i|_K = \delta_i$ for each $i$. Since each $g_e$ (from \eqref{P(t)}) is an $e$-th symmetric polynomial in $f_1, \ldots, f_d$, it follows that for all $i$, $1 \leq i \leq m$, and all $e$, $1 \leq e \leq d$,
\begin{align}
\delta_i(g_e) = \eta'_i(g_e) \leq e\eta'_i(f). \label{delta_i(g_e)}
\end{align}
Since $\eta'|_B = \eta$, it follows that $d' = \eta(f) = \eta'(f)$. By definition of $\eta'$, there exists $i$, $1 \leq i \leq m$, such that $d' = \eta'_i(f) \geq \eta'_{i'}(f)$ for all $i'$, $1 \leq i' \leq m$. It then follows from \eqref{delta_i(g_e)} that 
\begin{align}
\delta_i(g_e) \leq ed'\quad \text{for all}\ i,\ 1 \leq i \leq m. \label{delta_i(g_e)-again}
\end{align} 
Now recall that $A$ is integrally closed, so that $g_e \in A$ for all $e$ \cite[Proposition 5.15]{am}. Since inequality \eqref{delta_i(g_e)-again} implies that $\delta(g_e) \leq ed'$, it follows that $g_et^{ed'} \in \adelta$ for all $e$. Consequently $ft^{d'}$ satisfies the integral equation 
\begin{align*}
\tilde P(T) := T^d + \sum_e g_e t^{ed'} T^{d-e}
\end{align*}
over $\adelta$. Therefore $B^\eta$ is integral over $\adelta$, as required.
\end{proof}

\bibliographystyle{alpha}
\bibliography{bibi}

\end{document}

%% file: commands-2012.tex
\makeatletter

\newcommand{\Rmnum}[1]{\expandafter\@slowromancap\romannumeral #1@}
\makeatother

\DeclareMathOperator\proj{Proj}

\DeclareMathOperator\spec{Spec}

\newcommand{\scrB}{\ensuremath{\mathcal{B}}}

\newcommand{\scrG}{\ensuremath{\mathcal{G}}}

\newcommand{\scrP}{\ensuremath{\mathcal{P}}}

\newcommand{\cc}{\ensuremath{\mathbb{C}}}

\newcommand{\GG}{\ensuremath{\mathbb{G}}}
\newcommand{\kk}{\ensuremath{\mathbb{K}}}

\newcommand{\pp}{\ensuremath{\mathbb{P}}}

\newcommand{\rr}{\ensuremath{\mathbb{R}}}

\newcommand{\zz}{\ensuremath{\mathbb{Z}}}

\newcommand{\sheaf}{\ensuremath{\mathcal{O}}}

\newcommand{\nktorus}{(\kk^*)^n}

\newcommand{\dsum}{\ensuremath{\bigoplus}}
\newcommand{\into}{\ensuremath{\hookrightarrow}}

\newtheorem{thm}{Theorem}[section]
\newtheorem*{thm*}{Theorem}
\newtheorem{lemma}[thm]{Lemma}
\newtheorem*{lemma*}{Lemma}

\newtheorem*{prop*}{Proposition}
\newtheorem{cor}[thm]{Corollary}
\newtheorem{claim}[thm]{Claim}
\newtheorem*{claim*}{Claim}

\newtheorem*{conjecture*}{Conjecture}

\theoremstyle{definition}

\newtheorem*{constrinition*}{Construction-Definition}

\newtheorem*{convention*}{Convention}
\newtheorem{defn}[thm]{Definition}
\newtheorem*{defn*}{Definition}

\newtheorem*{definotation*}{Definition-Notation}
\newtheorem{example}[thm]{Example}
\newtheorem*{example*}{Example}

\newtheorem*{fact*}{Fact}

\newtheorem*{facts*}{Facts}

\newtheorem*{bold-note*}{Note}

\newtheorem*{problem*}{Problem}
\newtheorem{bold-question}[thm]{Question}
\newtheorem*{bold-question*}{Question}
\newtheorem{rem}[thm]{Remark}

\newtheorem*{reminition*}{Remark-Definition}

\newtheorem{remexample*}{Remark-Example}

\newtheorem*{remtation*}{Remark-Notation}
\newtheorem{remuestion}[thm]{Remark-Question}
\newtheorem*{remuestion*}{Remark-Question}

\newtheorem*{remvention*}{Remark-Convention}

\theoremstyle{remark}
\newtheorem*{rem*}{Remark}
\newtheorem*{note*}{Note}
\newtheorem*{notation*}{Notation}
\newtheorem*{question*}{Question}

\newtheorem*{questions*}{Questions}

\newcounter{UnorderedProofTempCtr}
\newcommand{\tempcommand}{}